\documentclass[reqno,a4paper,11pt]{amsart}

\usepackage{mystyle}
\usepackage{etoolbox}

\usepackage[headings]{fullpage}
\usepackage{parskip}
\usepackage{xy}
\xyoption{all}

\usepackage[utf8]{inputenc}
\usepackage[T1]{fontenc}
\usepackage[british]{babel}
\usepackage{microtype}
\usepackage{tikz-cd}

\usepackage[pdfencoding=auto,pdfusetitle]{hyperref}
\hypersetup{colorlinks=true, linkcolor=blue!30!black, citecolor=green!30!black, urlcolor=red!45!black}

\usepackage{amsfonts,amssymb,bbm,mathrsfs,stmaryrd}
\usepackage{amsmath}
\usepackage{mathtools}
\usepackage{centernot,url,booktabs}
\usepackage[noabbrev,capitalize]{cleveref}
\usepackage[shortlabels]{enumitem}

\usepackage{algorithm}
\usepackage[noend]{algpseudocode}

\usepackage{tikz}
\usetikzlibrary{matrix}
\usetikzlibrary{positioning}
\usetikzlibrary{calc}
\usetikzlibrary{arrows}
\tikzset{> =stealth}
\usetikzlibrary{arrows.meta}
\tikzset{normalHead/.tip={Triangle[open,angle=60:4pt]},}
\tikzset{normalTail/.tip={Triangle[reversed,open,angle=60:4pt]},}

\newcommand{\addQEDstyle}[2]{\AtBeginEnvironment{#1}{\pushQED{\qed}\renewcommand{\qedsymbol}{#2}}\AtEndEnvironment{#1}{\popQED}}

\theoremstyle{plain}
\newtheorem{theorem}{Theorem}[section]
\newtheorem{lemma}[theorem]{Lemma}
\newtheorem{proposition}[theorem]{Proposition}
\newtheorem{corollary}[theorem]{Corollary}

\theoremstyle{definition}
\newtheorem{definition}[theorem]{Definition}

\newtheorem{example}[theorem]{Example}
\addQEDstyle{example}{$\triangle$}

\newtheorem*{excont}{Example \continuation}
\addQEDstyle{excont}{$\triangle$}

\theoremstyle{remark}
\newtheorem{remark}[theorem]{Remark}
\addQEDstyle{remark}{$\triangle$}

\theoremstyle{definition}

\addQEDstyle{model}{$\triangle$}

\renewcommand{\epsilon}{\varepsilon}
\renewcommand{\phi}{\varphi}

\renewcommand{\S}{\mathbf{S}}
\newcommand{\X}{\mathbf{X}}
\newcommand{\G}{\mathbf{G}}
\newcommand{\Y}{\mathbf{Y}}
\newcommand{\M}{\mathbf{M}}

\newcommand{\op}{^\mathrm{op}}

\newcommand{\inv}{^{-1}}

\newcommand{\Hom}{\mathrm{Hom}}
\newcommand{\Surj}{\mathrm{Surj}}

\newcommand{\Iso}{\mathrm{Iso}}
\newcommand{\Aut}{\mathrm{Aut}}
\newcommand{\End}{\mathrm{End}}

\newcommand{\Grp}{\mathrm{Grp}}

\newcommand{\Inv}{\mathrm{Inv}}

\title{The Reverse Representation Problem}

\author[P. F. Faul]{Peter F. Faul}
\address{Stellenbosch University, Stellenbosch, South Africa}
\email{peter@faul.io}

\author[Z. Janelidze]{Zurab Janelidze}
\address{Stellenbosch University, Stellenbosch, South Africa}
\email{zurab@sun.ac.za}

\author[G. Joubert]{Gideo Joubert}
\address{Stellenbosch University, Stellenbosch, South Africa}
\email{22930701@sun.ac.za}

\date{\today}


\begin{document}
\maketitle
\thispagestyle{empty}

\begin{abstract}
Cayley's theorem tells us that all groups $\G$ occur as subgroups of the group of automorphisms over some set $X$. In this paper we consider a `sort-of' converse to this question: given a set $X$ and some transformation group $\S$ over $X$, what are the possible group structures on $X$ that result in groups represented by $\S$? We solve this problem in the more general setting of faithful semigroups and observe that the solutions to this problem, which we term \emph{unrepresentations}, have the structure of a heap. We study this phenomenon in depth and then move onto looking at particular classes of semigroups namely monoids, groups, inverse semigroups and Clifford semigroups.
\end{abstract}

\section{Introduction}

An abstract group $\X$ can be represented as the group $\S$ of transformations of the underlying set $X$, given by the mappings $y\mapsto xy$. According to Cayley's classical theorem \cite{cayley1854vii}, $\X$ is isomorphic to $\S$. This result generalises to the case where $\X$ is a faithful semigroup, i.e., a semigroup in which $xz=yz$ for all $z$, implies $x=y$. We call this corresponding transformation semigroup $\S$ the \emph{representation} of $\X$. A modern treatment of this topic can be found in \cite{jacobson2009CayleyRepresentation} and a version specific to semigroups can be found in \cite{Anderson2006-yu}. 

In this paper we are interested in a `sort-of' converse: given a transformation semigroup $\S$ of a set $X$, does it arise as a representation for some semigroup structure on $X$, and if yes, what can be said about the set of all such semigroup structures? Our goal is two-fold. Firstly, it is to solve the `unrepresentation' problem for semigroups. Secondly, we want to explain the links between the unrepresentation problem and the theory of heaps and torsors.

Regarding this first aim we show that the set of unrepresentations corresponds to a particular set of isomorphisms. A set of isomorphisms is different from a set of automorphisms in that the latter can be given the structure of a group while the former in general cannot be. We note that despite this deficiency of the set of isomorphisms it can nevertheless be given the structure of a heap in which one may define a ternary operation on the set of isomorphisms sending $f,g,h$ to $f \circ g\inv \circ h$. Heaps are themselves related to torsors and both may be thought of, in some sense, as corresponding to a group that has forgotten its identity element \cite{ZoranHeaps, Skorobogatov2001-yk}. This provides a number of valid perspectives to view this problem from.

We are able to use this result to show that when at least one unrepresentation exists, the set of unrepresentations can be made into a group isomorphic to the invertible elements of the centralizer of the transformation semigroup $\S$ in $\End(X)$. The caveat that at least one unrepresentation must exist is a big one and requires an investigation into when this occurs. Although we aren't able to give a complete characterization of such transformation semigroups we provide a necessary condition and prove that all cyclic transformation groups of appropriate size have unrepresentations.

We end by looking at results for particular classes of semigroups, specifically monoids, groups, inverse semigroups and Clifford semigroups. 

\section{The Reverse Representation Problem for Semigroups}

Let $\X = (X,\cdot)$ be a semigroup and $(\End(X),\circ)$ the associated function space with multiplication given by composition. We define a mapping $\phi_\X\colon \X \to (\End(X), \circ)$ with $\phi(x)(y) = x \cdot y$, which is readily seen to be a homomorphism. We often restrict $\phi_\X$ to its image $\S = \phi(X,\cdot)$ and call $\phi_\X$ the representing map and its image $\S$ the \emph{representation} of $\X$.

Given a transformation semigroup $\S \le (\End(X),\circ)$ there are two natural $\S$-actions. The action $\epsilon_\S$ of $\S$ on $X$ given by $f \cdot x = f(x)$ and the action $\circ_\S$ of $\S$ on itself given by composition.

\begin{proposition}
    Let $\X = (X,\cdot)$ be a semigroup and $\S$ its representation. Let $(X,\cdot)$ and $\S$ be equipped with their canonical $\S$-actions. Then $\phi_\X\colon (X,\cdot) \to \S$ is a surjective homomorphism of $\S$-actions.
\end{proposition}\label{prp:repact}

\begin{proof}
    It is immediate that $\phi_\X$ is surjective.
    We must show that $\phi_\X$ makes the following square commute:
    $$\xymatrix@=50pt{
\S\times X\ar[r]^-{1_S\times \phi_\X}\ar[d]_-{\varepsilon_\S} & \S\times S\ar[d]^-{\circ_\S}\\
X\ar[r]_-{\phi_\X} & S}$$

This will be so if $f(x) \cdot y = f(x \cdot y)$. Since $\phi_\X$ is a surjection $f = \phi(w)$ for some $w$ and hence we see that both sides of the equation reduce to $w \cdot x \cdot y$, completing the proof.
\end{proof}

\begin{definition}
    We denote the set of semigroups with underlying set $X$ and representation $\S$ by $\mathcal{S}(\S)$ and term it the set of \emph{unrepresentations} of $\S$. 
\end{definition}

We get the following simple corollary.

\begin{corollary}
    Let $X$ be a set and $\S \le (\End(X),\circ)$ a transformation semigroup. There exists a function $\ell\colon \mathcal{S}(\S) \to \Surj(\epsilon_\S,\circ_\S)$, sending a semigroup $\X$ to $\phi_\X$.
\end{corollary}\label{cor:fun}

With these somewhat trivial preliminaries out of the way we are ready to pose the question at the heart of this paper: Given a set $X$ and a transformation semigroup $\S \le (\End(X),\circ)$, what semigroup structures can be equipped to $X$ so that $(X,\cdot)$ is represented by $\S$. As is perhaps suggested by Proposition \ref{prp:repact}, the answer has to do with actions.

Let $\S \le (\End(X),\circ)$ be a transformation semigroup. Given a homomorphism $\phi$ of the $\S$-actions $\epsilon_\S$ and $\circ_\S$, it is possible to equip $X$ with a multiplication making $\phi$ a homomorphism of semigroups. When $\phi$ is surjective it will be the representing map of $\X$.

\begin{proposition}
    Let $X$ be a set, $\S \le (\End(X),\circ)$ a transformation semigroup and $\phi\colon X \to S$ a morphism of $\S$-actions as depicted below     
    $$\xymatrix@=50pt{
\S\times X\ar[r]^-{1_S\times \phi}\ar[d]_-{\varepsilon_\S} & \S\times S\ar[d]^-{\circ_\S}\\
X\ar[r]_-{\phi} & S}$$
Then the binary operation $\cdot \colon X \times X \to X$, with $x \cdot y = \phi(x)(y)$, makes $(X,\cdot)$ a semigroup, $\phi$ a semigroup homomorphism. Moreover, when $\phi$ is surjective it is the representing map of $(X,\cdot)$.
\end{proposition}\label{prp:actmul}

\begin{proof}
    We begin by proving that this operation is associative. Consider the equation
    \begin{align*}
        (x \cdot y) \cdot z &= \phi(\phi(x)(y))(z) \\
                            &= \phi(x)(\phi(y)(z) \\
                            &= x \cdot (y \cdot z)
    \end{align*}
    where the penultimate line follows from $\phi$ being a homomorphism of $S$-actions. Hence $(X,\cdot)$ is a semigroup.

    It remains to show that $\phi$ is a semigroup homomorphism. Since $\phi$ is a morphism of actions we have
    \begin{align*}
        \phi(x\cdot y)  &= \phi(\phi(x)(y)) \\
                        &= \phi(x) \circ \phi(y).
    \end{align*}
    Since $\phi$ is surjective its image is $\S$. It is immediate that $\phi$ is in fact the representation of $(X,\cdot)$ as by definition $\phi(x)(y) = x \cdot y$.
\end{proof}

\begin{corollary}
    Let $X$ be a set and $\S \le (\End(X),\circ)$ a transformation semigroup. There exists a function $k\colon \Surj(\epsilon_\S,\circ_\S) \to \mathcal{S}(S)$, sending an action $\phi$ to the semigroup $(X,\epsilon_\S \circ (\phi \times 1_X))$.
\end{corollary}

\begin{proposition}
 Let $X$ be a set and $S\le (\End(X),\circ)$ a transformation semigroup. Then the function $k\colon \Surj(\epsilon_\S,\circ_\S) \to \mathcal{S}(S)$, $\phi \mapsto (X,\epsilon_\S \circ (\phi \times 1_X))$ is a bijection.
\end{proposition}

\begin{proof}
    It suffices to show that $\ell$ from Corollary \ref{cor:fun} is the inverse of $k$. First consider 
    $\ell \circ k (\phi) = \ell (X,\epsilon_\S \circ (\phi \times 1_X))$ we know that $\ell$ returns the representation of $(X,\epsilon_\S \circ (\phi \times 1_X))$ and from Proposition $\ref{prp:actmul}$ we know that $\phi$ is the representation of $(X,\epsilon_\S \circ (\phi \times 1_X))$. Hence $\ell \circ k(\phi)= \phi$ as required.

    Finally consider $k \circ \ell (\X) = k(\phi_\X) = (X,\epsilon_\S \circ (\phi_\X \times 1_X))$ and note that 
    \begin{align*}
        \epsilon_\S \circ (\phi_\X \times 1_X)(x,y)   &= \epsilon_\S(\phi_\X(x),y) \\
                                                    &= \phi_\X(x)(y) \\
                                                    &= x \cdot y
    \end{align*}

    Thus the multiplications agree and hence $k \circ \ell (\X) = \X$.
\end{proof}

\begin{remark}
    A more general result can be proved where $\Hom(\epsilon_\S,\circ_\S)$ is shown to be in bijection with the set of semigroups on $X$ such that 
    \begin{enumerate}
\item The image of $\phi_\X$ lies in $\S$,
\item $s(x\cdot y)=s(x)\cdot y$ holds for all $s\in S$ and $x,y\in X$.
\end{enumerate}
    However we will have no cause to make use of this added generality in the rest of this paper.
\end{remark}

Hence we see that the set $\mathcal{S}(\S)$ of unrepresentations of $\S$ is simply the set of surjective action homomorphisms from $\epsilon_\S$ to $\circ_\S$. This perspective allows us to equip $\mathcal{S}(\S)$ with an algebraic operation for a natural class of semigroups.

\section{The Heap of Unrepresentations}

In the group and monoid case we can always rely on representing maps to be isomorphisms. For semigroups this is generally not the case as it is possible for two distinct elements $x$ and $y$ to be such that $xz = yz$ for all $z$. For this reason one often encounters the notion of faithful action of semigroups in the literature. In the introduction to \cite{Clifford1961-V1}, the authors provide a quick introduction to what they term `(left) representative' semigroups. These are semigroups in which the `(left) canonical representation' (this is exactly our notion of representation) is faithful (read injective). They prove that this happens if and only if the condition above is not satisfied. 
Due to `representative' being a somewhat overloaded term in the literature we name this class as follows.

\begin{definition}
    A semigroup $(X,\cdot)$ is \emph{faithful} if its representation is injective.
\end{definition}

In the case of faithful semigroups we can restrict to the image of the representation and the resulting subsemigroup $\S = \phi(X,\cdot)$ will necessarily be isomorphic to $(X,\cdot)$. Such a representation is said to be \emph{faithful}. From here on we will use representation to mean faithful representation.

\begin{corollary}
     Let $X$ be a set and $S\le (\End(X),\circ)$ a transformation semigroup. Then the function $k\colon \Iso(\epsilon_\S,\circ_\S) \to \mathcal{S}(S)$, $\phi \mapsto (X,\epsilon_\S \circ (\phi \times 1_X))$ is a bijection.
\end{corollary}

All monoids are faithful as the identity element suitably distinguishes all other elements in the sense above. But there are many natural classes of semigroups without identity that are faithful. Although we expect the below result to be well-known, for completeness we provide below a proof that inverse semigroups are faithful. (An introduction to inverse semigroups and their properties can be found in \cite{Clifford1967-V2})

\begin{proposition}
    Let $\X$ be an inverse semigroup. Then $\X$ is faithful.
\end{proposition}

\begin{proof}
    Let $x,y \in \X$ be such that for all $z \in \X$, $xz = yz$. We will show that $x$ and $y$ must be the same element by proving that $x\inv = y\inv$. By assumption we have that $x\inv y x\inv = x\inv x x\inv = x\inv$. A similar argument gives that $y\inv x y\inv = y\inv$. Taking inverses on both sides of the latter equation yields $y x\inv y = y$, which together with the former equation proves that $x\inv = y\inv$ as required.
\end{proof}

When we consider unrepresentations of faithful semigroups we can endow this set with the structure of a heap.

\begin{definition}
    A heap is a pair $(X,t)$ where $X$ is a non-empty set and $t\colon X \times X \times X \to X$ is a ternary operations satisfying the following equations.
        \begin{enumerate}
            \item $t(x,x,y) = y = t(y,x,x)$,
            \item $t(v,w,t(x,y,z)) = t(t(v,w,x),y,z)$.
    \end{enumerate}
    \end{definition}

The latter condition is something like an associativity condition, while the former is something like requiring inverse to exist. Indeed a heap is like a group that has forgotten its identity element. Below we demonstrate that the set of isomorphisms between two objects naturally forms a heap, though this result is known.

\begin{proposition}
    Let $\S$ be a transformation semigroup over some set $X$. Then $(\Iso(\epsilon,\circ),t)$ where $t(f,g,h) = f \circ g\inv h$ is a heap.
\end{proposition}

\begin{corollary}
    Let $\S$ be a transformation semigroup over some set $X$. Then the set $\mathcal{S}(\S)$ of unrepresentations can be made into a heap $(\mathcal{S}(\S),t)$ where if $\X_1 = (X,\cdot^1), \X_2 = (X,\cdot^2)$ and $\X_3 = (X,\cdot^3)$ are three unrepresentations of $S$, then $t(\X_1,\X_2,\X_3)$ is the unrepresentation with multiplication given by \[x\cdot y = \phi_{\X_2}\inv\phi_{\X_3}(x) \cdot^1 y.\]
\end{corollary}

We can learn a lot about the structure of this heap by studying the semigroup $\S$ itself. For instance by fixing an isomorphism $e \in \Iso(\epsilon,\circ)$ we can imbue $\Iso(\epsilon,\circ)$ with a group structure $f \cdot g = f e\inv g$ and identity $e$. The resulting group is isomorphic to $\Aut(\epsilon)$ via the mapping $\phi\colon \Aut(\epsilon) \to \Iso(\epsilon,\circ), f \mapsto ef$.

\begin{theorem}
    Let $\S$ be a transformation semigroup over the set $X$. Then when the set of unrepresentations is non-empty, the unrepresentations of $\S$ form a group isomorphic to the invertible elements of the centralizer of $\S$ in $\End(X)$.
\end{theorem}

\begin{proof}
    We have established that the set of unrepresentations forms a group structure on $\Iso(\epsilon,\circ)$ after making a choice of identity. Moreover we know that $\Iso(\epsilon,\circ) \simeq \Aut(\epsilon)$. It remains to simply unpack the definition of an element of $\Aut(\epsilon)$.

    $$\xymatrix@=50pt{
\S\times X\ar[r]^-{1_S\times \phi}\ar[d]_-{\varepsilon_\S} & \S\times X\ar[d]^-{\epsilon_\S}\\
X\ar[r]_-{\phi} & X}$$

The above diagram commutes exactly when $\phi \circ f = f \circ \phi$ for all $f \in \S$. We see that such an $f$ will be precisely an invertible element of $C_{\End(X)}(\S)$ as required.
\end{proof}

In a similar vein we can establish that $\Iso(\epsilon,\circ) \simeq \Aut(\circ)$ when the set of unrepresentations is inhabited. The elements of this latter group are functions $\alpha\colon \S \to \S$ satisfying that $\alpha(f\circ g) = f \circ \alpha(g)$. This motivates the following definition.

\begin{definition}
    Let $\S$ be a semigroup. A pseudounit of $S$ is a bijective function $\alpha\colon \S \to \S$ satisfying that for all $x,y \in \S$, $\alpha(xy) = x\alpha(y)$.
\end{definition}

\begin{corollary}
    Let $\S$ be a a semigroup. The set $P(\S)$ of pseudounits forms a group with respect to composition.
\end{corollary}

\subsection{Existence of Unrepresentations}
We now have a number of results concerning the algebraic structure of unrepresentations when some unrepresentations exist. It is natural then to ask about necessary and sufficient conditions on a transformation semigroup $\S$, to ensure that an unrepresentation exists.

One simple condition is the following.

\begin{proposition}
    A transformation semigroup $\S$ over some set $X$ has an unrepresentation only if $|\S| = |X|$.
\end{proposition}

\begin{proof}
    An unrepresentation equips $X$ with a multiplication making it isomorphic to $\S$, hence as sets they must be the same size.
\end{proof}

As expected, this condition is not sufficient to guarantee unrepresentation as can be seen in the following counterexample.
\begin{example}
    Let $X=\{1,2,3,4\}$ and consider the following set of functions from $X$ to $X$.
    \begin{center}
        $S= \begin{cases}
        f(1,2)=1, &f(3,4)=3\\
        g(1,2)=2, &g(3,4)=3\\
        h(1,2)=1, &h(3,4)=4\\
        j(1,2)=2, &j(3,4)=4
    \end{cases}$
    \end{center}
    One can easily see that for all $m,n\in S$, $m\circ n=m$ and so $\S$ is a faithful transformation semigroup of $X$ with 4 elements (it is a left-zero band). We will now prove that $\S$ has no unrepresentations, assume that we have an unrepresentation and consider the unrepresenting map $\phi\colon X\to S$. Then
    $$\phi(1)=\phi(f(1))=f\circ\phi(1)=f$$
    but then similarly, $\phi(1)$ should equal $h \ne f$. This contradiction demonstrates that $\S$ has no unrepresentations.
\end{example}

One can use the technique in this example to show that a left-zero band of transformations of $X$ has an unrepresentation if and only if it is the semigroup consisting of all constant functions on $X$. In that case, the unrepresentation will be the left-zero band on $X$.

In general the existence of an unrepresentation depends not only on the semigroup $\S$ but also on how exactly it is embedded into $\End(X)$ and so, there is no simple way to determine whether an unrepresentation exists. There is however at least one other instance in which knowledge of the structure of the semigroup $\S$ is enough to deduce the existence of an unrepresentation, without any reference to an embedding.

\begin{proposition}
    Let $\S = \{p^n:n \in \mathbb{Z}\}$ be a cyclic transformation group over the set $X$, such that the permutation $p$ comprises of a single cycle. Then $\S$ has an unrepresentation.
\end{proposition}

\begin{proof}
    To define an action we simply choose some distinguished element $z \in X$. Since $p$ comprises of a single cycle, for all $x \in X$ there exists some smallest positive $k$ such that $x = p^k(z)$. Now we simply define the function $\phi\colon X \to S$ to send $p^k(z) \mapsto p^k$.

    To show that $\phi$ defines an actions and consequently determines an unrepresentation of $S$ we must demonstrate that $\phi(p^n(x)) = p^n \circ \phi(x)$. Let $k\in \mathbb{Z}$ be the smallest positive $k$ such that $x = p^k(z)$. Then we have the following.

    \begin{align*}
        \phi(p^n(x))    &= \phi(p^n(p^k(z))) \\
                        &= \phi(p^{n+k}(z)) \\
                        &= p^{n+k} \\
                        &= p^n \circ p^k \\
                        &= p^n \circ \phi(x).
    \end{align*}

    It is easy to see that $\phi$ is both injective and surjective and so $\phi$ is an isomorphism of actions and hence corresponds to an unrepresentation.
\end{proof}

\section{Classes of Faithful Semigroups}

In this section we regard specific instances of faithful semigroups and observe what simplifications occur in these settings.

\subsection{Groups and Monoids}

The following result about pseudounits for monoids establishes that pseudounits were aptly named.

\begin{proposition}
    Let $\textbf{M}$ be a monoid. The group of pseudounits $P(\textbf{M})$ is dually isomorphic to the group of invertible elements of $\textrm{Inv}(\textbf{M})$.
\end{proposition}

\begin{proof}
    Notice that a pseudounit $\alpha$ is uniquely determined by where it sends $1$ since $\alpha(x) = x\alpha(1)$. Since $\alpha$ necessarily has an inverse $\alpha\inv$ it must be that $\alpha(1)$ is an invertible element as we need $\alpha(1)\alpha\inv(1) = \alpha\inv(\alpha(1)1) = \alpha\inv(\alpha(1)) = 1$.

    This defines an assignment $k \colon P(\textbf{M}) \to \mathrm{Inv}(\textbf{M})$ sending $\alpha$ to $\alpha(1)$. Note that $k(\alpha \beta) = \alpha(\beta(1)) =  \alpha(\beta(1)1) = \beta(1) \alpha(1)=k(\beta)k(\alpha)$.

    It remains to show that $k$ is bijective. It is clear that $k$ is injective and to see that it is surjective we need only establish that if $x \in \mathrm{Inv}(\textbf{M})$ then $\alpha(y) = yx$ defines a pseudounit. It is clear the identity is satisfied and because $x$ is invertible it is bijective. Hence $k$ must be surjective, completing the proof.
\end{proof}

\begin{corollary}\label{cor:invgrp}
    Let $\textbf{M}$ be a transformation monoid over the set $X$. The group of unrepresentations $\mathcal{S}(\textbf{M})$ is isomorphic to $\mathrm{Inv}(\textbf{M})$, the group of invertible elements of $\mathbf{M}$ when it is non-empty.
\end{corollary}

\begin{corollary}
    Let $\textbf{G}$ be a transformation group over the set $X$. The group of unrepresentations $\mathcal{S}(\textbf{G})$ is isomorphic to $\textbf{G}$ when it is non-empty.
\end{corollary}

\begin{corollary}
    Let $G = \{p^n:n \in \mathbb{Z}\}$ be a cyclic transformation group over the set $X$, such that the permutation $p$ comprises of a single cycle. Then the group of unrepresenations is cyclic.
\end{corollary}

Another nice simplification in the monoid setting is that in some sense an unrepresentation is entirely determined by what is sent to the identity.

\begin{proposition}\label{prp:monid}
        Any unrepresentation $\phi$ of a transformation monoid $\S$ over some set $X$ is completely determined by which element is mapped to $1$. Specifically for any $f\in \S$ we have
    \[\phi^{-1}(f)=f(\phi^{-1}(1))\]
\end{proposition}

\begin{proof}
    Note that we have that $f = f \circ 1 = f \circ \phi\phi\inv(1)$. Since $\phi$ is a morphism of actions we have that $f \circ \phi\phi\inv(1) = \phi(f(\phi\inv(1)))$. Hence we have that $\phi\inv(f) = \phi\inv(\phi(f(\phi\inv(1)))) = f(\phi\inv(1))$.
\end{proof}

Since $\phi$ is assumed to be an isomorphism, knowing the behaviour of the inverse tells us everything about $\phi$ itself. This result in fact generalises to the context of inverse semigroups which we explore in the next section.

We have yet to explore what the known connection between torsors and heaps gives us in this context.

\begin{definition}
    Let $\G$ be a group. A $\G$-torsor is an action $\alpha\colon \G \times X \to X$ such that the map $\G \times X \to X \times X, (g,x) \mapsto (g \cdot x,x)$ is an isomorphism.
\end{definition}

There is a bijective correspondence between heaps and torsors. Given a heap $(H,t)$ we can make $H$ a group by selecting an identity $e\in H$ and defining $x \cdot y = t(x,e,y)$. Then the map $\alpha\colon \mathbf{H}\times H \to H, (x,y) \mapsto x \cdot y$ is a torsor.

We may now ask what properties we expect of the torsor of unrepresenations. Let $M$ be a transformation monoid with at least one unrepresentation. By Corollary \ref{cor:invgrp} we have that the associated torsor is the map $\cdot\colon \Inv(\textbf{M}) \times M \to M, (x,y) \mapsto x \cdot y$.

In fact there is a canonical homomorphism of actions of this torsor $\alpha$ into $\epsilon$. 

    $$\xymatrix@=50pt{
\Inv(\mathbf{M})\times \Inv(M)\ar[r]^-{i\times \beta}\ar[d]_-{\cdot} & \mathbf{M}\times X\ar[d]^-{\epsilon_{\mathbf{M}}}\\
\Inv(M)\ar[r]_-{\beta} & X}$$

Here $i\colon \Inv(\mathbf{M}) \to \mathbf{M}$ is the inclusion and $\beta(\phi) = \phi\inv(1)$. The square is readily seen to commute.

In the case of groups we find that the associated torsor is isomorphic to $\epsilon$.

\subsection{Clifford and Inverse Semigroups}

Just as unrepresentations of monoids are uniquely determined by where $\phi\inv$ sends the identity, we have an analogous result for inverse semigroups. 
\begin{theorem}\label{ThmB}
    Any unrepresentation $\phi$ of an inverse semigroup of transformations $\S$ over a set $X$ is completely determined by which elements get mapped to idempotents. Specifically for any $f\in \S$ and idempotent $e\geq f\inv f$
    $$\phi^{-1}(f)=f(\phi^{-1}(e))$$
    
\end{theorem}
\begin{proof}
    The proof proceeds as in Proposition \ref{prp:monid}, except our starting point is that $f = f \circ e$.
\end{proof}

A particularly nice class of inverse semigroups are the Clifford semigroups whose idempotents are central elements (\cite{CliffordSemigroups}). It is well-known that Clifford semigroups may be thought of as being comprised of groups centered at each idempotent in the sense that a Clifford semigroup $\G$ corresponds to a functor $F\colon L\op \to \Grp$ where $L$ is the semilattice of idempotents \cite{pasku2011clifford}. Here $F(e) = \{x \in \G:xx\inv = e\}$ and $F(e \le e')\colon F(e') \to F(e)$ is the homomorphism which sends an element $x$ in $F(e')$ to $ex$. 

This raises the natural question of whether unrepresenations of Clifford semigroups may be thought of as being composed of unrepresenations of the constituent groups in some manner. Indeed this is the case, though we must introduce some new machinery before we can state the result.

Up until now, we have identified the unrepresentation of some transformation semigroup $\mathbf{S}$ of $X$ by a map $\phi:X\to S$. We now want to relax this requirement by only requiring that $\phi$ be a bijection on some subset $Y\subseteq X$. 

\begin{definition}
Let $\S$ be a transformation semigroup over some set $X$ and $Y \subseteq X$ such that $\epsilon_{Y}$, the restriction of $\epsilon$ to the domain $S\times Y$, satisfies that $\mathrm{Im}(\epsilon_{Y})=Y$. Then if $\phi$ is an isomorphism making the following diagram commutes we call the resulting semigroup $\Y$ with $y_1 \cdot y_2 = \phi(y_1)(y_2)$ an \emph{underrepresentation} of $\S$.
\begin{center}
    \begin{tikzcd}
\S\times Y \arrow[rr, "1\times \phi", shift left] \arrow[dd, "\epsilon_{Y}"] &  & \S\times S \arrow[ll, dashed, shift left] \arrow[dd, "\circ"] \\
                                                                         &  &                                                              \\
Y \arrow[rr, "\phi", shift left]                                           &  & S \arrow[ll, dashed, shift left]                            
\end{tikzcd}
\end{center}
\end{definition}

Since the full Clifford semigroup will be a transformation semigroup over the full set $X$, we will require this notion in order to talk about the `smaller' unrepresentations for each constituent group.

\begin{definition}
    Let $\S$ be a transformation semigroup over some set $X$ and $Y \subseteq X$ such that $\mathrm{Im}(\epsilon_Y) = Y$. Then we may define $f_Y\colon Y \to Y, y \mapsto f(y)$ and call the resulting collection $\S_Y = {f_Y:f\in \S}$ the \emph{deflation} of $\S$ with respect to $Y$.
\end{definition}

It is easy to check that the $\mathrm{Im}(\epsilon_Y) = Y$ condition ensures that each $f_Y$ is well defined. It is also not hard to see that the deflation $\S_Y$ is itself a semigroup with $f_Y \circ g_Y = (f\circ g)_Y$.

These notions' utility regarding the Clifford semigroup question follows from the fact that any unrepresentation of $S$ induces an underrepresentation on any subsemigroup of $S'$.

\begin{lemma}\label{LemB}
    Let $\S$ be a transformation semigroup over some set $X$ and $\phi$ the action homomorphism corresponding to some unrepresentation. If $\S'$ is a subsemigroup of of $\S$ and if $Y=\phi^{-1}(S')$ then the domain restriction $\phi_{Y}:Y\to S'$ is an underrepresentation of $\mathbf{S}'$
\end{lemma}
\begin{proof}
    Assume $\mathbf{S}$ is a translation semigroup of $X$ with an unrepresentation given by $\phi:X\to S$. We need to show that
    \begin{enumerate}
        \item $\phi_{Y}:Y\to S'$ is a bijection,
        \item  $\epsilon_{Y}:\S'\times Y\to Y$ is a well defined function, and
        \item $\phi_{Y}(\epsilon_{Y}(s,x))=\circ(1\times \phi_{Y}(s,x))$.
    \end{enumerate}
     The first is true by definition and the third will hold by restriction of the commutative diagram defining the unrepresentation $\phi$. To check (2) assume that $s_1\in \S'$ and $y\in Y$. By the definition of $Y$, we can write $y=\phi^{-1}(s_2)$ for some $s_2\in \S'$.  Now
        $$\epsilon_{Y}(s_1,y) = \epsilon(s_1,\phi^{-1}(s_2)) = \phi^{-1}(\phi(\epsilon(s_1,\phi^{-1}(s_2)))) =\phi^{-1}(s_1\circ s_2)$$
        and since $s_1s_2\in \S'$ and $Y=\phi^{-1}(S')$, we know that $\phi^{-1}(s_1s_2)\in Y$, this completes the proof.
\end{proof}

We are now ready to characterise the unrepresentation of a Clifford semigroup of transformations by the unrepresentations of these `component groups'. By Lemma~\ref{LemB} each unrepresentation of a Clifford semigroup will induce an underrepresentation of every component group that makes up the Clifford semigroup. Thus, one can easily deconstruct an unrepresentation into pieces. The question becomes how to identify whether these components of these piecewise unrepresentations will have any interaction with each other and to determine when one can construct an unrepresentation out of some compatible collection of underrepresentations.

\begin{theorem}\label{ThmC}
    Consider a transformation Clifford semigroup $\mathbf{S}$ over some set $X$ corresponding to the functor $F\colon L \to \Grp$, and a bijection  $\phi:X\to S$. Then $\phi$ is an unrepresentation of $\mathbf{S}$ if and only if for every $e\in L$, $\phi_e$ (the restriction of $\phi$ to $Y_e=\phi^{-1}(F(e))$) is an underrepresentation of $F(e)$ with the additional property that the square
    \begin{center}
        \begin{tikzcd}
Y_e \arrow[rr, "\phi_e", shift left] \arrow[dd, "f(-)"] &  & F(e) \arrow[ll, dashed] \arrow[dd, "F(f \le e)"] \\
                                                     &  &                                            \\
Y_f \arrow[rr, "\phi_f", shift left]                    &  & F(f) \arrow[ll, dashed]                    
\end{tikzcd}
    \end{center}
    commutes for all idempotents $e,f\in L$ satisfying $f\leq e$.
\end{theorem}
\begin{proof}
    \begin{itemize}
        \item [($\Rightarrow$)] Let $\phi$ be an unrepresentation of $\mathbf{S}$.  By Lemma~\ref{LemB}, $\phi_e$ is an underrepresentation of $F(e)$ for any $e\in L$. We need only check that the square commutes when $f\le e$. First we establish that $f(Y_e)\subseteq Y_f$, assume that $y_e\in Y_e$.
            \begin{align*}
            f(y_e)&=\epsilon(f,y_e)\\
            &= \phi^{-1}(\phi(\epsilon(f,y_e)))\\
            &= \phi^{-1}(f\circ \phi(y_e))
            \end{align*}
        We know that $\phi(y_e)\in F(e)$ by definition and further that $f\circ \phi(y_e)=F(f\le e)(\phi(y_e))\in F(f)$. Thus $\phi^{-1}(f\circ \phi(y_e))=f(y_e)\in Y_f$ and $f(Y_e)\subseteq Y_f$. Now to show commutativity. Let $e,f\in E$ such that $f\leq e$ and let $y_e\in Y_e$.
            \begin{align*}
            \phi_f(f(y_e))&= \phi(\epsilon(f,y_e))\\
            &= \circ(f,\phi(y_e))\\
            &= f\circ\phi(y_e)\\
            &= F(f \le e)(\phi(y_e)).
            \end{align*}
        Thus, the square commutes, completing this part of the proof.
        \item[($\Leftarrow$)]Let $\phi$ be a map such as described in the theorem statement, we need to show that $\phi$ is an unrepresentation map for $\mathbf{S}$. 
        
        By $g_x$ we denote an element in $F(x)$ and by $y_x$ we refer to an element in $Y_x$. We first show that for all $g_e\in S$ and $y_f\in X$, $g_e(y_f)\in Y_{ef}$. Let $g_e\in \S$ and $y_f\in X$.
            \begin{align*}
                g_e(y_f)&= e(g_e)(f(y_f))\\
                &= e\circ g_e\circ f(y_f)\\
                &= fg_e(e(y_f)) &&\text{Idempotents commute}\\
                &= F(ef \le e)(g_e)(e(y_f))
            \end{align*}
        Now notice that $F(ef \le e)(g_e)\in G_{ef}$ and $e(y_f)\in Y_{e f}$ so, $F(ef \le e)(g_e)(e(y_f))=g_e(y_f)\in Y_{ef}$. Notice that by assumption we know that the diagrams
            \begin{center}
               \begin{tikzcd}
Y_{ef}\times F(ef) \arrow[dd, "\varepsilon"] \arrow[rr, "1\times \phi_{ef}", shift left] &     & F(ef)\times F(ef) \arrow[dd, "\circ"] \arrow[ll, dashed] \\
                                                                                         & (1) &                                                          \\
Y_{ef} \arrow[rr, "\phi_{ef}", shift left]                                               &     & F(ef) \arrow[ll, dashed]                                
\end{tikzcd}
\begin{tikzcd}
Y_{ef} \arrow[rr, "\phi_{ef}", shift left]               &     & F(ef) \arrow[ll, dashed]                           \\
                                                         & (2) &                                                    \\
Y_f \arrow[rr, "\phi_f", shift left] \arrow[uu, "e(-)"'] &     & F(f) \arrow[ll, dashed] \arrow[uu, "F(ef \le f)"']
\end{tikzcd}
            \end{center}
        commute and, additionally $f\circ g_e\in F(ef)$ and $e(y_f)\in Y_{ef}$. To prove that the unrepresentation square for $\phi$ commutes, we need to show that $\phi_{ef}(g_e(y_f))=g_e\circ \phi_f(y_f)$.
            \begin{align*}
                \phi_{ef}(g_e(y_f))&=\phi_{ef}(e\circ g_e\circ f(y_f))\\
                &=\phi_{ef}(f\circ g_e\circ e(y_f))\\
                &=\phi_{ef}(f\circ g_e(e(y_f)))\\
                &=\phi_{ef}(\epsilon(f\circ g_e,e(y_f)))\\
                &= (f\circ g_e)\circ \phi_{ef}(e(y_f))\\
                &= (f\circ g_e)\circ(F(ef \le f)(\phi_f(y_f)))\\
                &= f\circ g_e\circ e\circ \phi_f(y_f)\\
                &= e\circ g_e\circ f\circ \phi_f(y_f)\\
                &= g_e\circ \phi_f(y_f)
            \end{align*}
        Thus, $\phi(\epsilon(g_e,y_f))=g_e\circ \phi(y_f)$ for all $g_e\in \S$ and $y_f\in X$. This completes the proof.
    \end{itemize}
\end{proof}

This result is still limited in the sense that it only teaches us something about the structure of the seimgroup of unrepresenations when at least one unrepresentation exists. We can however determine whether an unrepresenation exists, at least in the Clifford monoid case, by looking at the constituent groups as we demonstrate below.

\begin{lemma}
    A Clifford monoid $\M$ of transformations (of $X$) will have an unrepresentation if and only if each component group has an underrepresentation.

    In this case, the unrepresenting maps of $\M$ will be given by evaluation at $y$ for every $y\in Y_{1_X}$.
\end{lemma}
\begin{proof}
    The `only if' part of this lemma is trivial thus, we only need to show that every $F(e)\subseteq\M$ having an underrepresentation implies that $\M$ has an unrepresentation.

    Since every $F(e)$ has an underrepresentation, $F(1)$ (the group associated with the idempotent $1_X$) has an underrepresentation. By Proposition~\ref{prp:monid} every underrepresenting map for $F(1)$ will be of the form
    $$\epsilon_y:F(1)\to Y_1,$$
    we will show that $\epsilon_y:M\to X$ is an unrepresenting map for $\M$. We claim that $\epsilon_y=\epsilon_{e(y)}:F(e)\to Y_e$ is an underrepresenting map for every $e\in \mathcal{L}$. Indeed, for any $g_e\in F(e)$
    $$\epsilon_{e(y)}(g_e)=g_e(e(y))=[g_e\circ e](y)=g_e(y)=\epsilon_y(g_e).$$
    Now, since $F(e)$ is guaranteed to have an underrepresentation (and is a group) $\epsilon_{y_e}:F(e)\to Y_e$ is an underrepresentation for any $y_e\in Y_e$. The element $e(y)$ is in $Y_e$ for every $e\in\mathcal{L}$ so, we have a family  of underrepresentation maps $\epsilon_y:F(e)\to Y_e|e\in\mathcal{L}$. To finish, we only need to show the compatibility condition:
     \begin{center}
        \begin{tikzcd}
Y_e \arrow[dd, "f(-)"] \arrow[rr, dashed] &  & F(e) \arrow[ll, "\epsilon_y"', shift right] \arrow[dd, "F(f \le e)"] \\
                                          &  &                                                                   \\
Y_f \arrow[rr, dashed]                    &  & F(f) \arrow[ll, "\epsilon_y"', shift right]                    
\end{tikzcd}
    \end{center}
However, notice that $F(f \le e)=f\circ -$ and as such, the square trivially commutes. This completes the proof.
\end{proof}

\begin{corollary}
    If $\G$ is a Clifford monoid associated to a functor $F$ in which each $F(e) = \{p^n: n \in \mathbb{Z}\}$ is a cyclic group in which the permutation $p$ comprises of a single cycle, then an unrepresentation exists.
\end{corollary}

\bibliographystyle{abbrv}
\bibliography{bibliography}
\end{document}